\documentclass[12pt, a4paper]{amsart}
\usepackage{amsmath, amsthm, amstext, amsbsy, amssymb}
\usepackage{mathrsfs}

\usepackage[backend=biber, style=alphabetic, sorting=nyt, nohashothers=true, maxbibnames=99, maxalphanames=99, giveninits=true]{biblatex}
\usepackage{xcolor}                 
\usepackage[shortlabels]{enumitem} 

\usepackage[a4paper, margin=2.5cm]{geometry}

\usepackage{todonotes} 

\usepackage{hyperref} 				
\hypersetup{						
	colorlinks=true,
	linkcolor=blue,
	citecolor=red,
}

\allowdisplaybreaks[1]		   			

\newtheorem{prop}{Proposition}
\newtheorem{lem}{Lemma}
\newtheorem{thrm}{Theorem}

\newtheorem*{thrm*}{Theorem}

\newtheorem{rmk}{Remark}

\theoremstyle{definition}  
\newtheorem{definition}{Definition}

\renewcommand{\qedsymbol}{$ \blacksquare $}

\renewcommand{\le}{\leqslant}
\renewcommand{\ge}{\geqslant}

\let\intt\int
\renewcommand{\int}{\intt\limits}

\newcommand{\C}{\mathbb{C}} 						
\newcommand{\D}{\mathbb{D}} 						
\newcommand{\R}{\mathbb{R}}		                    
\newcommand{\N}{\mathbb{N}}	                	    
\newcommand{\TT}{\mathbb{T}}	 					
\newcommand{\Z}{\mathbb{Z}}	 					    
                        
\newcommand{\F}{\mathcal{F}}                        

\renewcommand{\a}{\alpha}									
\renewcommand{\b}{\beta}									
\newcommand{\g}{\gamma}						                
\newcommand{\G}{\Gamma}		                                
\newcommand{\de}{\delta}									
\newcommand{\e}{\varepsilon}								
\newcommand{\z}{\zeta}										
\renewcommand{\l}{\lambda}									
\renewcommand{\L}{\Lambda}									
\let\originalnu\nu											
\renewcommand{\nu}{\originalnu} 					        
\newcommand{\f}{\varphi}									

\DeclareMathOperator{\dist}{dist}							

\DeclareMathOperator{\Hol}{Hol}

\renewcommand{\emptyset}{\varnothing}

\title[Complete interpolating sequences for small Fock spaces]
{A uniform approach to complete interpolating sequences for small Fock spaces with $p > 0$}
\author{Mikhail Mironov}
\address{
\newline \phantom{x}\,\, Mikhail Mironov,
\newline Univ Gustave Eiffel, Univ Paris Est Créteil, CNRS, LAMA UMR8050 F-77447 Marne-la-Vallée, France
\newline {\tt mikhail.mironov2@univ-eiffel.fr}
\smallskip
}

\begin{document}

\begin{abstract} 

We study complete interpolating sequences in two types of small Fock spaces, $\mathcal{F}^p_{\alpha +}$ and $\mathcal{F}^p_{\alpha}$, for $0 < p \le \infty$. One-sided small Fock spaces $\mathcal{F}^p_{\alpha +}$ are well-studied spaces of entire functions with sub-exponential growth, while $\mathcal{F}^p_{\alpha}$ are their two-sided analogue with a symmetric singularity at the origin.

For \emph{one-sided small Fock spaces} $\mathcal{F}^p_{\alpha +}$, we provide a streamlined, perturbation-type description of complete interpolating sequences that unifies and extends earlier results for $1 \le p \le \infty$ to the full range $0 < p \le \infty$. 

For \emph{two-sided small Fock spaces} $\mathcal{F}^p_{\alpha}$, we establish a parallel characterization, revealing a curious periodicity phenomenon: complete interpolating sequences for $\mathcal{F}^p_{\alpha}$ coincide exactly for $p = 1$, $p = 2$, and $p = \infty$, but differ for other $p \ge 1$.

\end{abstract} 

\maketitle

\section{Introduction} 
    In this paper, we consider one-sided and two-sided small Fock spaces. We study the problem of characterizing complete interpolating sequences, which belongs to the broader theory of sampling and interpolation in analytic function spaces. 

    We start by defining one-sided and two-sided small Fock spaces, establish what it means for a sequence to be complete interpolating, and present an appropriate notion of distance in this setting. Next, we state the central result of this paper, a description of complete interpolating sequences, which is in the spirit of the famous 1/4 Kadets–Ingham theorem for the Paley-Wiener space.
    
    \subsection{Sampling and interpolation in Fock-type spaces}
        \begin{definition}
            Suppose $\f : \C \to \R$ is an increasing function on $[0, \infty)$ such that $\f(z) = \f(|z|)$. \emph{The Fock-type space} for this weight $\f$ and $0 < p \le \infty$ is
            \begin{align*}
                \F^{\infty}_{\f} & = \{ f \in \Hol(\C): \| f \|_{\F^{\infty}_{\f}} = \sup_{z \in \C} |f(z)|e^{-\f(z)} < \infty \}, \quad p = \infty, \\
                \F^p_{\f} & = \{ f \in \Hol(\C): \|f\|^p_{\F^p_{\f}} = \int_{\C} |f(z)|^pe^{-p\f(z)} dA(z) < \infty \}, \quad 0 < p < \infty,  
            \end{align*}
            where $dA$ is the two-dimensional Lebesgue measure. 
        \end{definition}
        The classical examples of Fock-type spaces are Bargmann-Fock spaces, obtained for $\f(z) = c |z|^2$. For them, the problems of sampling and interpolation were completely resolved by Seip \cite{SeipBargFock}, and Seip and Wallstén \cite{SeipBargFock2}. In particular, no complete interpolating sequences exist for Bargmann-Fock spaces. It turns out that similar results hold for a much wider class of weights that roughly satisfy $\f(z) \gg \log_+^2|z|$, see \cite{lyubarskii1994sampling, Ortega1995OnInterpolation, marco2003interpolating, borichev2007sampling}. Borichev and Lyubarskii \cite{FockType} showed that complete interpolating sequences for $\F^2_{\f}$ exist exactly for the weights with rates of growth below the threshold $\f(z) = \log^2_+|z|$. Fock-type spaces with such slowly growing weights are referred to as \emph{small Fock spaces}.

    \subsection{One-sided and two-sided small Fock spaces}
        In this paper, we are interested specifically in the critical weight $\f(z) = \a \log^2_+|z|, \ \a > 0$. We call Fock-type spaces with this weight \emph{one-sided small Fock spaces}, and denote them $\F^p_{\a+}$. 
        Complete interpolating sequences for one-sided small Fock spaces were described by Baranov, Dumont, Hartmann and Kellay \cite{SmallFock} for $p = 2, \infty$, and this description was extended to $1 \le p \le \infty$ by Omari \cite{OmariCIS}, also see \cite{baranov2017fock, OmariCISb} for related results. One of the main results of this paper is the simplification of this description, which allows us to extend it to the full range $0 < p \le \infty$. 
    
        When considering the logarithmic weight, it appears to be natural to symmetrize one-sided small Fock spaces. That is, we allow a singularity at the origin. This way, one side refers to the singularity at $\infty$, while the other side refers to the singularity at $0$.
        \begin{definition}
            \emph{The two-sided small Fock space} for $\a > 0$ and $0 < p \le \infty$ is
            \begin{equation*}
                \F^p_{\a} = \{ f \in \Hol(\C \setminus \{ 0 \} ): \|f\|_{\F^p_{\a}} = \| f(z) e^{-\a \log^2|z|} \|_{L^p(\C)} < \infty \}.
            \end{equation*}    
        \end{definition}
        Note that one-sided small Fock spaces are their subspaces (with equivalent norm) that consist of entire functions, and the spaces $\F^p_{\a}$ and $\F^p_{\a+}$ are Banach spaces for $1 \le p \le \infty$, while for $0 < p < 1$ they are F-spaces.
        From now on, by $\f(z)$ we mean either $\a \log^2_+ |z|$ or $\a\log^2|z|$ depending on whether we work in the one-sided or two-sided context.  
    
        Another reason to consider two-sided small Fock spaces is that they are closely related to certain problems in Gabor analysis, see Baranov and Belov \cite{HypCh}. Two-sided small Fock spaces also appear implicitly in the work of Baranov, Belov, and Gröchenig \cite{GaussFrames}, where the authors effectively obtain them by gluing together two copies of one-sided small Fock spaces.

    \subsection{Geometry and complete interpolation}    
        For brevity, we denote $\C_0 = \C \setminus \{ 0 \}$.
        
        We now introduce an appropriate notion of separation in our setting.  
        Observe that $\R \times \TT \cong \C_0$ via $(t, \z) \mapsto e^t \z$. Hence, one can consider a metric $d_{\log}$ on $\C_0$ --- the product metric on $\R \times \TT$, that is, for $z, w \in \C_0$
        \begin{equation*}
            d_{\log} (z, w) = | \log|z| - \log|w| | + \left| \frac{z}{|z|} - \frac{w}{|w|} \right|. 
        \end{equation*}
        \begin{definition}
            We say that $\L \subset \C_0$ is \emph{$d_{\log}$-separated} if there exists $d_{\L} > 0$ such that for any two different $\l, \l' \in \L$
            \begin{equation*}
                d_{\log}(\l, \l') \ge d_{\L}.
            \end{equation*}
        \end{definition}
        Note that for $|z|, |w| \ge 1$ the $d_{\log}$-metric is equivalent to
        \begin{equation*}
            d(z, w) = \frac{|z - w|}{1 + \min(|z|, |w|)}    
        \end{equation*}
        for sufficiently close (in either $d_{\log}$ or $d$) points $z, w$, see \cite[Section 1.2]{SmallFock} for the discussion underlying $d(z,w)$. In particular, away from $0$ the notions of separation for these two coincide. Hence, in the one-sided case we say that a set $\L \subset \C$ is \emph{$d_{\log_+}$-separated} if it is $d_{\log}$-separated outside of any disc centered at $0$ and Euclidean separated inside any such disc.
    
        For a set $\L \subset \C_0$, consider the restriction operator acting on $\F^p_{\a}, \ 0 < p \le \infty$ 
        \begin{equation} \label{eq: R two-sided}
            R_{\L}: f \mapsto \left( f(\l)|\l|^{2/p}e^{-\a \log^2|\l|} \right)_{\l \in \L}.
        \end{equation}
        Similarly, in the one-sided case, for a set $\L \subset \C$ consider the restriction operator on $\F^p_{\a+}, \ 0 < p \le \infty$
        \begin{equation} \label{eq: R one-sided}
            R_{\L}: f \mapsto \left( f(\l)(1+|\l|)^{2/p}e^{-\a \log_+^2|\l|} \right)_{\l \in \L}.
        \end{equation}
        The weights that appear in these operators correspond to the norms of the evaluation functionals on small Fock spaces; we prove it in Subsection \ref{subsect: eval}.
    
        We define the main object of our study.
        \begin{definition}
            We say that $\L \subset \C_0$ (resp. $\C$) is \emph{a complete interpolating sequence for} $\F^p_{\a}$ (resp. $\F^p_{\a+}$) if the restriction operator $R_{\L}$ is an isomorphism between $\F^p_{\a}$ (resp. $\F^p_{\a+}$) and $\ell^p(\L)$.
        \end{definition}

    \subsection{Characterization of complete interpolating sequences}
        Now, we can state the main results of this paper. We provide a uniform description of complete interpolating sequences for one-sided and two-sided small Fock spaces for all $0 < p \le \infty$. This description has a perturbative nature, resembling the famous $1/4$ Kadets–Ingham theorem \cite{kadets1964exact} and Avdonin's theorem \cite[Theorem 1]{avdonin1977riesz} for the Paley-Wiener space. The main difference is that the perturbation-type condition is not only sufficient but also necessary in the case of small Fock spaces, while in the Paley-Wiener setting the complete answer is quite more complicated, see \cite[Theorem 1]{lyubarskii1997complete} for the relevant discussion. 
        
        For brevity, we state the following theorem for sets $\L \subset \C_0$, since moving a single point does not affect the complete interpolating property.
        \begin{thrm} \label{thrm: CIS characterisation for one-sided}
            Let $0 < p \le \infty$. A set $\L \subset \C_0$ is complete interpolating for $\F^p_{\a+}$ if and only if the following three conditions are satisfied:
            \begin{enumerate} [label=(\roman*)]
                \item $\L$ is $d_{\log_+}$-separated. \label{pr: 1 one-sided}
            \end{enumerate}
            In particular, we can enumerate $\L = ( \l_k )_{k \ge 0}$ so that $|\l_k| \le |\l_{k + 1}|$ and write $$\l_k = e^{\frac{k + 2/p + \de_k}{2\a}} e^{i \theta_k}, \ \de_k, \theta_k \in \R.$$
            \begin{enumerate} [label=(\roman*)]
                \addtocounter{enumi}{1}
                \item $(\de_k)_{k \ge 0} \in \ell^{\infty}(\N_0)$ \label{pr: 2 one-sided}
                \item There exists $N \in \N$ and $\e > 0$ such that
                \begin{equation*}
                    \sup_{n \ge 0} \left| \frac{1}{N} \sum_{k = n}^{n + N - 1} \de_k \right| \le \frac{1}{2} - \e.
                \end{equation*} \label{pr: 3 one-sided}
            \end{enumerate}
        \end{thrm}
        Note that for $p = 2$ and $p = \infty$ we get \cite[Theorems 1.1 and 1.2]{SmallFock}. Moreover, for $1 \le p \le \infty$, after appropriate substitutions we obtain \cite[Theorems 1, 2]{OmariCIS}. In this way, Theorem \ref{thrm: CIS characterisation for one-sided} provides us with a uniform description of complete interpolating sequences for small Fock spaces $\F^p_{\a+}$ for all $1 \le p \le \infty$, and extends this description to $0 < p < 1$ in a natural way.
    
        We prove a similar result for two-sided small Fock spaces.
        \begin{thrm} \label{thrm: CIS characterisation for two-sided}
            Let $0 < p \le \infty$. A set $\L \subset \C_0$ is complete interpolating for $\F^p_{\a}$ if and only if the following three conditions are satisfied:
            \begin{enumerate} [label=(\roman*)]
                \item $\L$ is $d_{\log}$-separated. \label{pr: 1 two-sided}
            \end{enumerate}
            In particular, we can enumerate $\L = ( \l_k )_{k \in \Z}$ so that $|\l_k| \le |\l_{k + 1}|$ and write $$\l_k = e^{\frac{k + 2/p + \de_k}{2\a}} e^{i \theta_k}, \ \de_k, \theta_k \in \R.$$
            \begin{enumerate} [label=(\roman*)]
                \addtocounter{enumi}{1}
                \item $(\de_k)_{k \in \Z} \in \ell^{\infty}(\Z)$ \label{pr: 2 two-sided}
                \item Up to a shift in numeration there exists $N \in \N$ and $\e > 0$ such that 
                \begin{equation*}
                    \sup_{n \in \Z} \left| \frac{1}{N} \sum_{k = n}^{n + N - 1} \de_k \right| \le \frac{1}{2} - \e.
                \end{equation*} \label{pr: 3 two-sided}
            \end{enumerate}
        \end{thrm}
        \begin{rmk} \label{rmk: m shift}
            Note that instead of saying that there exists a shift in numeration we can explicitly say that there exists $m \in \Z$, $N \in \N$ and $\e > 0$ such that
            \begin{equation*}
                \sup_{n \in \Z} \left| \frac{1}{N} \sum_{k = n + 1}^{n + N} \de_k + m \right| \le \frac{1}{2} - \e.    
            \end{equation*}
            Moreover, such an $m$, if it exists, is unique and, hence, the right shift in numeration is also unique.
        \end{rmk}
        For $1 \le p \le \infty$, one way to prove this theorem is to realize $\F^p_{\a}$ as two copies of $\F^p_{\a+}$ appropriately glued together. This allows one to obtain Theorem \ref{thrm: CIS characterisation for two-sided} from Theorem \ref{thrm: CIS characterisation for one-sided} by the means of Fredholm theory, see the argument in \cite[Theorem  1.3]{GaussFrames}.
        
        In stark contrast to the Paley-Wiener case, there is no obstruction at $p = 1$, and the description even extends to $p > 0$; it also does not break down at $p = \infty$. This phenomenon is quite different from what happens in the Paley-Wiener spaces due to the difference in the behavior of the discrete Hilbert transform, see \cite{eoff1995discrete}. 
    
        The author finds it quite curious that if $0 < p, q \le \infty$ are such that $2/p - 2/q$ is an integer, then, due to Theorem \ref{thrm: CIS characterisation for two-sided}, complete interpolating sequences for $\F^p_{\a}$ and $\F^q_{\a}$ are exactly the same. We see that a certain type of periodicity in $2/p$ takes place in this description of complete interpolating sequences.  
        In particular, a sequence is complete interpolating for $\F^{\infty}_{\a}$ if and only if it is complete interpolating for $\F^{2}_{\a}$ if and only if it is complete interpolating for $\F^{1}_{\a}$.
    
    \subsection*{Notation} 
        $A \lesssim B$ (equivalently $B \gtrsim A$) means that there is a constant $C > 0$ independent of the relevant variables such that $A \le CB$. We write $A \asymp B$ if both $A \lesssim B$ and $A \gtrsim B$.
    
    \subsection*{Organization of the paper} 
        Section \ref{sect: preliminaries} contains preliminary results involving estimates of infinite products and the restriction operator on small Fock spaces. Section \ref{sect: Theorem one-sided} and Section \ref{sect: Theorem two-sided} are devoted to the proofs of Theorem \ref{thrm: CIS characterisation for one-sided} and Theorem \ref{thrm: CIS characterisation for two-sided} respectively. 
        
\section{Preliminaries} \label{sect: preliminaries}
    In this section, we establish several properties of small Fock spaces, as well as provide the necessary estimates for infinite products, which play a crucial role in the proofs of the main results.

    \subsection{Canonical products estimates}
        We start with a general estimate for the canonical products of sequences that grow at least exponentially. It can be found in the proof of \cite[Theorem 1.1]{SmallFock}. We present a proof here for the reader's convenience.
        \begin{lem} \label{lem: product estimates coarse}
            Let $\L \subset \C_0$ be a $d_{\log_+}$-separated sequence, enumerated $(\l_n)_{n \ge 0}$ such that $|\l_{n+1}| \ge |\l_n|$, then the canonical product
            \begin{equation*}
                G_{\L}(z) = \prod_{k \ge 0} \left( 1 - \frac{z}{\l_k} \right) 
            \end{equation*}
            satisfies the estimate
            \begin{equation*}
                |G_{\L}(z)| \asymp \frac{|z - \l_n|}{|\l_n|} \prod_{k = 0}^{n - 1} \left| \frac{z}{\l_k} \right|,
            \end{equation*}
            where $\l_n$ is the $d_{\log}$-closest point to $z$. Moreover, the multiplicative constants in this estimate are uniform as long as we bound the $d_{\log_+}$-separation constant of $\L$ from $0$.
        \end{lem}
        \begin{rmk}
            If we let $z \to \l_n$, we get an estimate for the derivative of $G_{\L}$:
            \begin{equation*}
                |G'_{\L}(\l_n)| \asymp \frac{1}{|\l_n|} \prod_{k = 0}^{n - 1} \left| \frac{\l_n}{\l_k} \right|.
            \end{equation*}
        \end{rmk}
        \begin{proof}
            Take the logarithm of $|G_{\L}(z)|$:
            \begin{equation*}
                \log|G_{\L}(z)| = \sum_{k \ge 0} \log \left| 1 - \frac{z}{\l_k} \right|.
            \end{equation*}
            We decompose the sum into three parts $\log|G_{\L}(z)| = A + B + C$, where 
            \begin{equation*}
                A = \sum_{k = 0}^{n - 1} \log \left| 1 - \frac{z}{\l_k} \right|, \qquad B = \log \left| 1 - \frac{z}{\l_n} \right|, \qquad C =  \sum_{k = n + 1}^{\infty} \log \left| 1 - \frac{z}{\l_k} \right|.
            \end{equation*}
            To complete the proof, it suffices to show that
            \begin{equation} \label{eq: A + C - boundedness}
                \left| A + C - \sum_{k = 0}^{n - 1} \log \left| \frac{z}{\l_k} \right| \right|
            \end{equation}
            is uniformly bounded. 
            
            First, we claim that $|C|$ is uniformly bounded. If we fix $\e > 0$, then for any $t: \: 0 \le |t| < 1 - \e$ we have $| \log (1 + t) | \le c_{\e} |t|$ with some $c_{\e} > 1$. To apply this asymptotic to $C$, note that since $\l_n$ is the $d_{\log}$-closest point to $z$, there exist $N > 0$ and $\e > 0$ depending only on the separation constant of $\L$ such that $\left| \frac{z}{\l_k} \right| < (1 - \e)^{k - n - N}, \ k > n + N$. Therefore,
            \begin{equation*}
                |C| \le \sum_{k = n + 1}^{n + N} \left| \log \left| 1 - \frac{z}{\l_k} \right| \right| + c_{\e } \sum_{k = n + N + 1}^{\infty} \left| \frac{z}{\l_k} \right| \le \sum_{k = n + 1}^{n + N} \left| \log \left| 1 - \frac{z}{\l_k} \right| \right| + \frac{c_{\e} (1 - \e)}{\e}.  
            \end{equation*}
            Finally, note that for $n + 1 \le k \le n + N$, $\left| 1 - \frac{z}{\l_k} \right|$ is uniformly bounded away from $0$ and $\infty$. Thus, $|C|$ is uniformly bounded as we claimed. 
    
            Now, the second claim is that 
            \begin{equation*}
                \left| A - \sum_{k = 0}^{n - 1} \log \left| \frac{z}{\l_k} \right| \right|    
            \end{equation*}
            is uniformly bounded. 
            To this end, note that
            \begin{equation*}
                \left| A - \sum_{k = 0}^{n - 1} \log \left| \frac{z}{\l_k} \right| \right| = \left| \sum_{k = 0}^{n - 1} \log \left| 1 - \frac{\l_k}{z} \right| \right|,
            \end{equation*}
            which is uniformly bounded by the same argument that we used for $C$.
    
            We conclude that \eqref{eq: A + C - boundedness} is uniformly bounded and the lemma is proven. 
        \end{proof}
    
        For a sequence $\L$ that resembles a geometric sequence more closely, we can get a finer estimate. The following lemma is essentially a perturbed version of \cite[Lemma 2.6]{FockType}.
        \begin{lem} \label{lem: product estimate fine}
            Let $\L \subset \C_0$ be a $d_{\log_+}$-separated sequence, enumerated $(\l_n)_{n \ge 0}$ such that $|\l_{n+1}| \ge |\l_n|$. Moreover, $\l_n = e^{\frac{n + \de_n}{2 \a}} e^{i \theta_n}, \ \de_n, \theta_n \in \R$ with $(\de_n)_{n \ge 0} \in \ell^{\infty}(\N_0)$. Then we have the estimate
            \begin{equation*}
                |G_{\L}(z)| \asymp \frac{e^{\f(z)}\dist(z, \L)}{(1 + |z|)^{\frac 1 2 + \frac{1}{n + 1} \sum_{k = 0}^{n} \de_k}}, 
            \end{equation*}
            where $\l_n$ is the $d_{\log}$-closest point to $z$ and $\f(z) = \a \log^2_+|z|$. 
        \end{lem}
        \begin{proof}
            The desired estimate holds for $|z| < r$ with any fixed $r > 0$. Hence, it is enough to prove
            \begin{equation*}
                |G_{\L}(z)| \asymp \frac{e^{\f(z)}\dist(z, \L)}{|z|^{\frac 1 2 + \frac{1}{n + 1} \sum_{k = 0}^{n} \de_k}}, \quad |z| \ge r, 
            \end{equation*}
            where $r > 1$ is chosen so that $n \ge 1$ for any $z: \: |z| \ge r$. Let us proceed to the proof. Lemma \ref{lem: product estimates coarse} gives us
            \begin{equation*}
                |G_{\L}(z)| \asymp \frac{|z - \l_n|}{|\l_n|} \prod_{k = 0}^{n - 1} \left| \frac{z}{\l_k} \right|.
            \end{equation*}
            Since $(\de_n)_{n \ge 0} \in \ell^{\infty}(\N_0)$ we have $\dist(z, \L) \asymp |z - \l_n|$, so that
            \begin{equation*}
                |G_{\L}(z)| \asymp \frac{\dist(z, \L)}{|\l_n|} \prod_{k = 0}^{n - 1} \left| \frac{z}{\l_k} \right|.
            \end{equation*}
            Thus, we need to show
            \begin{equation*}
                \frac{1}{|\l_n|} \prod_{k = 0}^{n - 1} \left| \frac{z}{\l_k} \right| \asymp \frac{e^{\f(z)}}{|z|^{\frac 1 2 + \frac{1}{n + 1} \sum_{k = 0}^{n} \de_k}}, \quad |z| \ge r.
            \end{equation*}
            We take the logarithm on both sides, so it remains to show the boundedness of
            \begin{equation*}
                L = \left| n \log|z| - \sum_{k = 0}^n \log |\l_k| - \f(z) + \frac 1 2 \log|z| + \frac{\log|z|}{n + 1} \sum_{k = 0}^{n} \de_k \right|.
            \end{equation*}
            Indeed,
            \begin{align*}
                L &= \left| n \log|z| - \sum_{k = 0}^n \frac{k + \de_k}{2 \a} - \a \log^2|z| + \frac 1 2 \log|z| + \frac{\log|z|}{n + 1} \sum_{k = 0}^{n} \de_k \right| \\
                &= \left| n \log|z| - \frac{n(n+1)}{4\a} - \frac{1}{2 \a}\sum_{k = 0}^n \de_k - \a \log^2|z| + \frac 1 2 \log|z| + \frac{\log|z|}{n + 1} \sum_{k = 0}^{n} \de_k \right| \\
                &= \left| - \a \left( \log |z| - \frac{n}{2\a} \right)^2 + \frac 1 2 \left( \log|z| - \frac{n}{2 \a} \right) + \left( \frac{1}{n + 1} \sum_{k = 0}^{n} \de_k \right) \left( \log|z| - \frac{n + 1}{2 \a} \right) \right|
            \end{align*}
            is bounded, since $\log|z| - \frac{n}{2 \a}$ and $\frac{1}{n + 1} \sum_{k = 0}^{n} \de_k$ are bounded due to $(\de_n)_{n \ge 0} \in \ell^{\infty}(\N_0)$. This completes the proof.
        \end{proof}
        
    \subsection{Evaluation estimates} \label{subsect: eval}
        The following lemma provides crucial estimates on the growth of functions in small Fock spaces. 
    
        \begin{lem} \label{lem: evaluation estimate}
            We have the following estimates for the value at a point
            \begin{equation*}
                |f(z)| \lesssim \frac{e^{\f(z)}}{(1 + |z|)^{2/p}} \|f\|_{\F^p_{\a+}}, \quad f \in \F^p_{\a+}, \ z \in \C,
            \end{equation*}
            \begin{equation*}
                |f(z)| \lesssim \frac{e^{\f(z)}}{|z|^{2/p}} \|f\|_{\F^p_{\a}}, \quad f \in \F^p_{\a}, \ z \in \C_0,
            \end{equation*}
            with $\f(z) = \a \log^2_+|z|$ and $\f(z) = \a \log^2|z|$ respectively.
        \end{lem}
        The factors that appear in this lemma are exactly the weights in the definition of the restriction operators $R_{\L}$, \eqref{eq: R two-sided}, \eqref{eq: R one-sided}. Moreover, considering $f(z) = z^n$, it can be argued, similarly to \cite[Lemma 2.7]{FockType}, that these expressions are equivalent to the appropriate ``norms'' of the evaluation functionals in small Fock spaces.
    
        Lemma \ref{lem: evaluation estimate} follows from the following more general result. It can be found in a similar form in \cite[Lemma 2.1]{SmallFock} and a more general form on the unit disc with the proof in \cite[Lemma 4.1]{borichev2007sampling}.
    
        \begin{lem} \label{lem: more accurate estimates}
            For sufficiently small $r > 0$, and for any holomorphic function $f$, the following estimates hold: 
            \begin{enumerate}
                \item For any $z_0 \in \C_0$, and any $z \in \C_0$ with $d_{\log}(z, z_0) < r$,
                \begin{equation*}
                    \left| |f(z)|e^{-\f(z)} - |f(z_0)|e^{-\f(z_0)} \right| \lesssim d_{\log}(z, z_0) \max_{d_{\log}(z, w) < 2r} |f(w)e^{-\f(w)}|.
                \end{equation*} 
                \item
                \begin{equation*}
                    |f(z)|e^{-\f(z)} \lesssim \frac{1}{|z|^2} \int_{d_{\log}(w, z) < r} |f(w)|e^{-\f(w)} dA(w).
                \end{equation*}
            \end{enumerate}
        \end{lem}
        The proof in \cite[Lemma 4.1]{borichev2007sampling} automatically implies the following.
        \begin{rmk} \label{rmk: evaluation estimate}
            Let $0 < p < \infty$ be arbitrary. For sufficiently small $r > 0$ and any holomorphic function $g$
            \begin{equation*}
                |g(z)|^p e^{-p\f(z)} \lesssim \frac{1}{|z|^2} \int_{d_{\log}(w, z) < r} |g(w)|^pe^{-p\f(w)} dA(w).
            \end{equation*}
        \end{rmk}
        Lemma \ref{lem: evaluation estimate} follows from this remark.    

\section{Proof of Theorem \ref{thrm: CIS characterisation for one-sided}} \label{sect: Theorem one-sided}
    The general idea is that the main arguments from \cite[Theorems 1.1]{SmallFock} and \cite[Theorem 1]{OmariCIS} remain valid even for $0 < p < 1$. First, we prove the sufficiency of the conditions \ref{pr: 1 one-sided}, \ref{pr: 2 one-sided}, \ref{pr: 3 one-sided} and, second, we prove their necessity. Before proceeding to the proof, recall that in the one-sided context $\f(z)$ stands for $\a \log^2_+|z|$.

    \subsection{Sufficiency}
        We need to prove that $\L$ satisfying properties \ref{pr: 1 one-sided}, \ref{pr: 2 one-sided}, \ref{pr: 3 one-sided} is complete interpolating for $\F^p_{\a+}$. We first prove the simpler ``complete'' part and then tackle the more complicated ``interpolating'' part.
        \subsubsection{Completeness}
            Let us show that any $\L$ that satisfies properties \ref{pr: 1 one-sided}-\ref{pr: 3 one-sided} is a uniqueness set for $\F^p_{\a+}$. We prove this by contradiction. 
            Suppose that there is a non-zero function $f \in \F^p_{\a+}$ such that $f|_{\L} = 0$. That implies a factorization $f(z) = h(z)G_{\L}(z)$ with some entire function $h(z)$ and the canonical product
            \begin{equation*}
                G_{\L}(z) = \prod_{k \ge 0} \left( 1 - \frac{z}{\l_k} \right).
            \end{equation*}
            Lemma \ref{lem: product estimate fine} gives us an estimate for $|G_{\L}|$ from below
            \begin{equation*}
                |G_{\L}(z)| \gtrsim \frac{e^{\f(z)} \dist(z, \L)}{(1 + |z|)^{1 + 2/p - \e}}.
            \end{equation*}
            Then, estimating $f$ with Lemma \ref{lem: evaluation estimate}, we get
            \begin{equation*}
                |h(z)| \lesssim \frac{e^{\f(z)}}{(1+|z|)^{2/p}} \frac{(1 + |z|)^{1 + 2/p - \e}}{e^{\f(z)} \dist(z, \L)} = \frac{(1 + |z|)^{1 - \e}}{\dist(z, \L)}.
            \end{equation*}
            We can consider contours $C_n$, $n \ge 0$ around $0$ such that $C_n \cap e^n \D = \emptyset$ on which an estimate $\dist(z, \L) \gtrsim 1 + |z|$ holds. Hence, for a fixed $z \in \C$, by the maximum principle, for large enough $n$ we get
            \begin{equation*}
                |h(z)| \lesssim \max_{w \in C_n} \frac{(1 + |w|)^{1 - \e}}{\dist(w, \L)} \lesssim e^{-\e n}.
            \end{equation*}
            Letting $n$ tend to $\infty$, we conclude that $h(z) = 0$, which implies $f = 0$, a contradiction. 
        
        \subsubsection{Base sequence}
            Before going into the proof of the interpolating property for an arbitrary $\L$, let us analyze the base unperturbed sequence
            \begin{equation*}
                \G = ( \g_k )_{k \ge 0} = \left( e^{\frac{k + 2/p}{2\a}} \right)_{k \ge 0}.
            \end{equation*}
            The respective generating function is
            \begin{equation*}
                G_{\G}(z) = \prod_{k \ge 0} \left( 1 - \frac{z}{\g_k} \right).
            \end{equation*}
            \begin{prop} \label{prop: Gamma CIS one-sided}
                The sequence $\G$ is complete interpolating for $\F^p_{\a+}$. Moreover, it stays complete interpolating even if we change the phases of $( \g_k )_{k \ge 0}$ in any way.
            \end{prop}
            \begin{proof}
                Note that we have already proven the fact that $\G$ is a uniqueness set for $\F^p_{\a+}$, since $\G$ is $\L$ with all $\de_k = 0$. Hence, we need to prove that $\G$ is an interpolating sequence for $\F^p_{\a+}$.
            
                For $k \ge 0$ set
                \begin{equation*}
                    g_k(z) = \frac{G_{\G}(z)}{G'_{\G}(\g_k)(z - \g_k)}.
                \end{equation*}
                Suppose that we have a sequence $(c_k)_{k \ge 0} \in \ell^p(\N_0)$. We claim that 
                \begin{equation} \label{eq: interpolation formula for gamma}
                    f(z) = \sum_{k \ge 0} c_k (1+|\g_k|)^{-2/p}e^{\f(\g_k)} g_k(z)
                \end{equation}
                belongs to $\F^p_{\a+}$ and interpolates this sequence $c$ in the sense that $(R_{\G} f)(\g_k) = c_k, \ k \ge 0$. We prove this claim. First, we estimate $|G_{\G}|$ using Lemma \ref{lem: product estimate fine}:
                \begin{equation} \label{eq: Generating for Gamma estimate}
                    |G_{\G}(z)| \asymp \frac{e^{\f(z)} \dist(z, \G)}{(1 + |z|)^{1/2 + 2/p}}, \quad z \in \C,
                \end{equation}
                and, in particular,
                \begin{equation} \label{eq: Generating for Gamma derivates}
                    |G'_{\G}(\g)| \asymp \frac{e^{\f(\g)}}{(1 + |\g|)^{1/2 + 2/p}}, \quad \g \in \G.
                \end{equation}
                From \eqref{eq: Generating for Gamma derivates} it follows that the series \eqref{eq: interpolation formula for gamma} converges absolutely and is bounded for bounded $z$, hence $f$ is a well-defined entire function. Moreover, from the definition of $g_k, \ k \ge 0$ we conclude $f(\g_k) = c_k$. Thus, it remains to show that $f \in \F^p_{\a+}$. For $p \ge 1$, this was shown in \cite[Theorem 2.6]{OmariCIS}. Here, we first present the argument for $0 < p < 1$ and then for the sake of completeness provide a short proof for $p \ge 1$. For $0 < p < 1$ we can use sub-additivity $|x + y|^p \le |x|^p + |y|^p$ to get an estimate
                \begin{equation*}
                    \|f\|^p_{\F^p_{\a+}} = \int_{\C} |f(z)|^p e^{-p\f(z)} dA(z) \le \sum_{k \ge 0} |c_k|^p \int_{\C} (1+|\g_k|)^{-2}e^{p\f(\g_k)} |g_k(z)|^p e^{-p\f(z)} dA(z).
                \end{equation*}
                Thus, to show $f \in \F^p_{\a+}$ it is sufficient to prove
                \begin{equation*}
                    I_k :=  \int_{\C} (1+|\g_k|)^{-2}e^{p\f(\g_k)} |g_k(z)|^p e^{-p\f(z)} dA(z) \lesssim 1.    
                \end{equation*}
                Using \eqref{eq: Generating for Gamma estimate} and \eqref{eq: Generating for Gamma derivates} to estimate $g_k(z)$ we get
                \begin{equation} \label{eq: estimate on g_k}
                    |g_k(z)| \asymp \frac{\dist(z, \G)e^{\f(z)}}{(1 + |z|)^{1/2 + 2/p} |z - \g_k|} \frac{(1 + |\g_k|)^{1/2 + 2/p}}{e^{\f(\g_k)}},
                \end{equation}
                hence,
                \begin{equation*}
                    (1+|\g_k|)^{-2}e^{p\f(\g_k)} |g_k(z)|^p e^{-p\f(z)} \asymp 
                    \frac{\dist(z, \G)^p (1 + |\g_k|)^{p/2}}{(1 + |z|)^{p/2 + 2} |z - \g_k|^p},
                \end{equation*}
                which implies
                \begin{equation*}
                    I_k \asymp \int_{\C} \frac{\dist(z, \G)^p (1 + |\g_k|)^{p/2}}{(1 + |z|)^{p/2 + 2} |z - \g_k|^p} dA(z).
                \end{equation*}
                To estimate this integral, we partition the complex plane $\C$ into two parts $D_1 = \{z: |z| < |\g_k|/2 \}$, $D_2 = \{z: |\g_k|/2 \le |z| \}$, which decomposes $I_k$ into the sum of respective integrals $I_k = I_k^{(1)} + I_k^{(2)}$. We then estimate each of them separately:
                \begin{align*}
                    I_k^{(1)} &\asymp \int_{|z| < |\g_k|/2} \frac{(1 + |z|)^p (1 + |\g_k|)^{p/2}}{(1 + |z|)^{p/2 + 2} (1 + |\g_k|)^p} dA(z) \asymp (1 + |\g_k|)^{-p/2} \int_0^{|\g_k|/2} (1 + r)^{p/2 - 2} r dr \lesssim 1, \\
                    I_k^{(2)} &\lesssim \int_{|\g_k|/2 \le |z|} \frac{(1 + |\g_k|)^{p/2}}{(1 + |z|)^{p/2 + 2}} dA(z) \asymp (1 + |\g_k|)^{p/2} \int_{|\g_k|/2}^{\infty} (1 + r)^{-p/2 - 2} r dr \lesssim 1. 
                \end{align*} 
                We conclude that $I_k \lesssim 1$, and so $f \in \F^p_{\a+}$, which completes the proof for the case $0 < p < 1$.

                Now, let us provide an outline for the proof in the case $p \ge 1$. Since we do not have sub-additivity, the approach is slightly different. We bound $|f(z)|$ given by \eqref{eq: interpolation formula for gamma}, using the estimate on $|g_k(z)|$ \eqref{eq: estimate on g_k} that is valid for all $p \ge 0$:
                \begin{equation} \label{eq: estimate on f p > 1}
                    |f(z)|e^{-\f(z)} \lesssim  \sum_{k \ge 0} |c_k|  \frac{\dist(z, \G)(1 + |\g_k|)^{1/2}}{(1 + |z|)^{1/2 + 2/p} |z - \g_k|} ,
                \end{equation}
                We claim that
                \begin{equation} \label{eq: weight equivalent to 1}
                    \sum_{k \ge 0} \frac{\dist(z, \G)(1 + |\g_k|)^{1/2}}{(1 + |z|)^{1/2} |z - \g_k|} \asymp 1.
                \end{equation}
                The fact that this sum is $\gtrsim 1$ is evident by considering the $d_{\log}$-closest to $z$ point $\g_n$ and taking only the $k = n$ term. Then, the bound $\lesssim 1$ follows from $\dist(z, \G) \lesssim 1 + |z|$ and considering separately the sum for $k \le n$ and $k > n$, estimating $|z - \g_k| \asymp 1 + |z|$ and $1 + |\g_k|$, respectively. For $p = \infty$, \eqref{eq: weight equivalent to 1} directly gives $f \in \F^p_{\a+}$. For $1 \le p < \infty$, we use Jensen's inequality to estimate $|f(z)|^p$:
                \begin{equation*}
                    |f(z)|^pe^{-p\f(z)} \lesssim \sum_{k \ge 0} |c_k|^p (1 + |z|)^{-2} \frac{\dist(z, \G)(1 + |\g_k|)^{1/2}}{(1 + |z|)^{1/2} |z - \g_k|},
                \end{equation*}
                from where we get
                \begin{equation*}
                    \|f\|^p_{\F^p_{\a+}} = \int_{\C} |f(z)|^pe^{-p\f(z)} dA(z) \lesssim \sum_{k \ge 0} |c_k|^p \int_{\C } \frac{\dist(z, \G)(1 + |\g_k|)^{1/2}}{(1 + |z|)^{5/2}|z - \g_k|}dA(z).
                \end{equation*}
                Thus, if we define
                \begin{equation*}
                    I_k = \int_{\C } \frac{\dist(z, \G)(1 + |\g_k|)^{1/2}}{(1 + |z|)^{5/2} |z - \g_k|}dA(z),
                \end{equation*}
                it is sufficient to prove $I_k \lesssim 1$. The idea is the same as before; we partition the complex plane $\C$ into two parts $D_1 = \{z: |z| < |\g_k|/2 \}$, $D_2 = \{z: |\g_k|/2 \le |z| \}$, and estimate $\frac{\dist(z, \G)}{|z - \g_k|} \lesssim \frac{1 + |z|}{1 + |\g_k|}$ on $D_1$ and $\frac{\dist(z, \G)}{|z - \g_k|} \lesssim 1$ on $D_2$. Integrating separately over $D_1$ and $D_2$ with these estimates we get $I_k \lesssim 1$. This completes the proof in the case $1 \le p \le \infty$.

                We have proven that $\G$ is complete interpolating for $\F^p_{\a+}$. Finally, note that we are free to add any phases to $\g_k$ without affecting the proof.            
            \end{proof}

        \subsubsection{Interpolation} 
            Now, let us prove that an arbitrary $\L$ satisfying properties \ref{pr: 1 one-sided}-\ref{pr: 3 one-sided} is also interpolating for $\F^p_{\a+}$. Notice that since $\L$ is $d_{\log_+}$-separated, by Remark \ref{rmk: evaluation estimate} the restriction operator $R_{\L}: \F^p_{\a+} \to \ell^p(\L)$ is bounded. Consider 
            \begin{equation*} 
                g_k(z) = \frac{G_{\L}(z)}{G'_{\L}(\l_k)(z - \l_k)}.        
            \end{equation*}
            By Lemma \ref{lem: product estimate fine} and property \ref{pr: 3 one-sided} we have
            \begin{equation*}
                \frac{e^{\f(z)} \dist(z, \L)}{(1 + |z|)^{1 + 2/p - \e}} \lesssim |G_{\L}(z)| \lesssim \frac{e^{\f(z)} \dist(z, \L)}{(1 + |z|)^{2/p + \e}}, \quad z \in \C,    
            \end{equation*}  
            \begin{equation} \label{eq: estimate on G' Lambda}
                e^{\f(\l)} (1 + |\l|)^{-1 - 2/p + \e} \lesssim |G'_{\L}(\l)| \lesssim e^{\f(\l)} (1 + |\l|)^{- 2/p - \e}, \quad \l \in \L.     
            \end{equation}
            which implies that for a fixed $k \ge 0$ we can estimate $|g_k(z) e^{-\f(z)}| \lesssim (1 + |z|)^{-2/p - \e}, \ z \in \C$. In particular, $g_k \in \F^p_{\a+}$. This means that we can interpolate finite sequences $c \in \ell^p(\N_0)$ by finite linear combinations of functions $g_k$ as follows:
            \begin{equation*}
                f(z) = \sum_{k \ge 0} c_k (1+|\l_k|)^{-2/p}e^{\f(\l_k)} g_k(z).
            \end{equation*}
            Suppose that we can prove stability on finite sequences, i.e., $\|f\| \lesssim \|c\|$. Then by passing to the limit, which exists and is an entire function by \eqref{eq: estimate on G' Lambda}, and using Fatou's lemma, we get the desired interpolation property for $\L$. 
            
            Now, let us prove stability on finite sequences. The main idea is to use the already proven Proposition \ref{prop: Gamma CIS one-sided} which says that $\G$ is complete interpolating for $\F^p_{\a+}$. This means that instead of proving $\|f\| \lesssim \|c\|$ for finite sequences $c$, we can prove $\|R_{\G}f\| \lesssim \|c\|$ for finite sequences $c$. This is equivalent to proving the continuity of the Hilbert transform-type operator $T: \ell^p(\N_0) \to \ell^p(\N_0)$
            \begin{equation} \label{eq: T definition}
                (Tc)_m = \sum_{k \ge 0} \left( \frac{1+|\g_m|}{1+|\l_k|} \right)^{2/p} e^{\f(\l_k) - \f(\g_m)} \frac{G_{\L}(\g_m)}{G'_{\L}(\l_k)(\g_m - \l_k)} c_k = \sum_{k \ge 0} T_{mk} c_k, 
            \end{equation}
            where $(T_{mk})_{m, k \ge 0}$ is the matrix corresponding to $T$. To understand $T$, we consider the matrix elements.
            \begin{equation*}
                |T_{mk}| \asymp \left| \frac{\g_m}{\l_k} \right|^{2/p} e^{\f(\l_k) - \f(\g_m)} \left| \frac{G_{\L}(\g_m)}{G'_{\L}(\l_k)(\g_m - \l_k)} \right|.
            \end{equation*}
            It follows from Lemma \ref{lem: product estimates coarse} and property \ref{pr: 2 one-sided} that 
            \begin{align*}
                |G_{\L}(\g_m)| & \asymp \frac{\dist(\g_m, \L)}{|\g_m|} \prod_{j = 0}^{m - 1} \frac{|\g_m|}{|\l_j|}, \\
                |G'_{\L}(\l_k)| & \asymp \frac{1}{|\l_k|} \prod_{j = 0}^{k - 1} \frac{|\l_k|}{|\l_j|}.
            \end{align*}
            Thus, we estimate
            \begin{equation*}
                |T_{mk}| \asymp e^{\f(\l_k) - \f(\g_m)} \left| \frac{\g_m}{\l_k} \right|^{2/p - 1} \frac{\dist(\g_m, \L)}{|\g_m - \l_k|} \prod_{j = 0}^{m - 1} \frac{|\g_m|}{|\l_j|} \prod_{j = 0}^{k - 1} \frac{|\l_j|}{|\l_k|}.
            \end{equation*}
            Since $\L$ is fixed, let us choose phases for $\G$ in such a way that $\dist(\g_m, \L) \asymp |\g_m|$. Then we get
            \begin{equation} \label{eq: Tmk estimate}
                |T_{mk}| \asymp 
                \begin{cases}
                    1, \quad m = k, \\
                    e^{\frac{1}{4 \a} \left( -(m - k) - 2 \sum_{j = k}^{m - 1} \de_j \right)}, \quad m > k, \\
                    e^{\frac{1}{4 \a} \left( -(k - m) + 2 \sum_{j = m}^{k - 1} \de_j \right)}, \quad m < k.
                \end{cases}
            \end{equation}
            Finally, due to property \ref{pr: 3 one-sided}, there exists a sufficiently large $M > 0$ such that 
            \begin{equation*}
                \left| \sum_{j = k}^{m - 1} \de_j \right| \le \left( \frac 1 2 - \frac \e 2 \right) (m - k), \quad m - k \ge M, 
            \end{equation*}
            which means that
            \begin{equation*}
                |T_{mk}| \lesssim e^{- \frac{\e}{4 \a} |m - k|}, \quad m, k \ge 0.
            \end{equation*}
            It follows that the operator $T$ is continuous from $\ell^p(\N_0)$ onto itself, which is what we needed to complete the proof.

    \subsection{Necessity}
        Suppose that $\L$ is complete interpolating. We prove properties \ref{pr: 1 one-sided}, \ref{pr: 2 one-sided}, \ref{pr: 3 one-sided}.

        \subsubsection{Property \ref{pr: 1 one-sided}}
            Suppose that $\L$ is not $d_{\log_+}$-separated and let us arrive at a contradiction. First, we show that $\L \cap r\D$ is Euclidean separated. If not, there is a sequence of distinct point pairs $\l_{n_j}, \l_{m_j} \in \L \cap r\D$ such that $
            |\l_{n_j} - \l_{m_j}| \to 0$. Consider $f_j$ to be the unique solution to the interpolation problem $f(\l) = \de_{\l, \l_{n_j}}$, so that $\|f_j\|_{\F^p_{\a+}} \lesssim 1$.
            Evaluating $f_j$ at the points $\l_{n_j}, \l_{m_j}$ and taking $j \to \infty$ we arrive at a contradiction to the fact that $|f'| \lesssim \|f\|_{\F^p_{\a+}}$ inside $r \D$, which follows from the mean value property. A similar argument shows that $\L \setminus r \D$ is $d_{\log}$-separated. We consider a sequence of distinct point pairs $\l_{n_j}, \l_{m_j} \in \L \setminus r\D$ such that $d_{\log}(\l_{n_j}, \l_{m_j}) \to 0$. Consider $f_j$ to be the solution to the interpolation problem $f(\l) = e^{\f(\l)} |\l|^{-2/p} \de_{\l, \l_{n_j}}$, so that $\|f_j\|_{\F^p_{\a+}} \lesssim 1$. Evaluating $f_j$ at the points $\l_{n_j}, \l_{m_j}$ and taking $j \to \infty$ we arrive at a contradiction to Lemma \ref{lem: more accurate estimates}.
            
            Hence, we have property \ref{pr: 1 one-sided} and so we can enumerate $\L = ( \l_k )_{k \ge 0}$ so that $|\l_k| \le |\l_{k + 1}|$ and write 
            \begin{equation*}
                \l_k = e^{\frac{k + 2/p + \de_k}{2\a}} e^{i \theta_k}, \quad \de_k, \theta_k \in \R.    
            \end{equation*}
            Finally, property \ref{pr: 1 one-sided} allows us to choose phases for $\G$ in such a way that $d_{\log}(\l_k, \G) \asymp 1$. 

        \subsubsection{Property \ref{pr: 2 one-sided}}
            For property \ref{pr: 2 one-sided} we need to show that $(\de_k)_{k \ge 0} \in \ell^{\infty}(\N_0)$. Since $\L$ is complete interpolating that implies that for any $k \ge 0$ there exists $g_k \in \F^p_{\a+}$ that uniquely interpolates $g_k(\l_j) = \de_{k,j}, \ j \ge 0$, moreover $\|g_k\|^p_{\F^p_{\a+}} \asymp (1 + |\l_k|)^2 e^{-p \f(\l_k)}$. It follows from Lemma \ref{lem: evaluation estimate} that $g_k$ is of order $0$, therefore, by the Hadamard factorization theorem, we have
            \begin{equation*}
                g_k(z) = \prod_{j \ge 0, j \ne k} \left( 1 - \frac{z}{\l_j} \right).
            \end{equation*}
            In other words,
            \begin{equation} \label{eq: biorthogonal element one-sided}
                g_k(z) = \frac{G_{\L}(z)}{G'_{\L}(\l_k)(z - \l_k)}.        
            \end{equation}
            Now, for contradiction, suppose that $\de_k$ is not bounded. Then there exists a subsequence $k_j$ such that $|\de_{k_j}| \to \infty$. If we let $\g_{m_j} \in \G$ be the $d_{\log}$-closest point to $\l_{k_j}$, then we get $|m_j - k_j| \to \infty$. Our next step is to study the behavior of $|g_{k_j}(\g_{m_j})|$. It follows from Lemma \ref{lem: product estimates coarse} and property \ref{pr: 1 one-sided} that
            \begin{align*}
                |G_{\L}(\g_{m_j})| & \asymp \prod_{l = 0}^{k_j - 1} \frac{|\g_{m_j}|}{|\l_l|}, \\
                |G'_{\L}(\l_{k_j})| & \asymp \frac{1}{|\l_{k_j}|} \prod_{l = 0}^{k_j - 1} \frac{|\l_{k_j}|}{|\l_l|}.
            \end{align*}
            From this it follows that
            \begin{equation*}
                |g_{k_j}(\g_{m_j})| \asymp \left| \frac{\g_{m_j}}{\l_{k_j}} \right|^{k_j} \asymp e^{\frac{1}{4 \a} 2k_j (m_j - k_j - \de_{k_j})}.
            \end{equation*}
            At the same time, by Lemma \ref{lem: evaluation estimate}, we have
            \begin{equation*}
                |g_{k_j}(\g_{m_j})| \lesssim \frac{e^{\f(\g_{m_j})}}{|\g_{m_j}|^{2/p}} \| g_{k_j} \|_{\F^p_{\a+}} \asymp e^{\f(\g_{m_j}) - \f(\l_{k_j})} \asymp e^{\frac{1}{4 \a}\left( (m_j + 2/p)^2 - (k_j + \de_{k_j} + 2/p)^2 \right)}.
            \end{equation*}
            Combining these two estimates, we get
            \begin{equation} \label{eq: de to infty contradiction}
                e^{\frac{1}{4 \a} (m_j - k_j - \de_{k_j})(k_j - m_j - \de_{k_j} - 4/p)} \lesssim 1.
            \end{equation}
            By choosing a subsequence of $\de_{k_j}$, we can assume that $\de_{k_j} \to \infty$ or $\de_{k_j} \to -\infty$. In the case $\de_{k_j} \to \infty$ we get $k_j - m_j - \de_{k_j} - 4/p \to -\infty$. So, it is enough to decrease all $m_j$ by one sufficiently large integer to get $m_j - k_j - \de_{k_j} \le -1$ to arrive at a contradiction to \eqref{eq: de to infty contradiction}. Similarly, in the case $\de_{k_j} \to -\infty$, we get $k_j - m_j - \de_{k_j} - 4/p \to \infty$. So, increasing all $m_j$ by one sufficiently large integer to get $m_j - k_j - \de_{k_j} \ge 1$ gives us a contradiction to \eqref{eq: de to infty contradiction}. We conclude that $\de_k$ has to be bounded, which is exactly property \ref{pr: 2 one-sided}.

        \subsubsection{Property \ref{pr: 3 one-sided}}
            Finally, let us prove property \ref{pr: 3 one-sided}. Suppose that it does not hold, i.e., for every $N > 0$ 
            \begin{equation*}
                \sup_{n \ge 0} \left| \frac{1}{N} \sum_{k = n}^{n + N - 1} \de_k \right| = \frac 1 2 + \e_N,     
            \end{equation*}
            with $\e_N \ge 0$. Consider again the operator $T$ as in \eqref{eq: T definition}.
            Since $\L$ is complete interpolating, $T$ must be continuous. Moreover, because we established property \ref{pr: 2 one-sided}, the estimate \eqref{eq: Tmk estimate} holds true, i.e.,
            \begin{equation*}
                |T_{mk}| \asymp 
                \begin{cases}
                    1, & \quad m = k, \\
                    e^{\frac{1}{4 \a} \left( -(m - k) - 2 \sum_{j = k}^{m - 1} \de_j \right)} ,& \quad m > k, \\
                    e^{\frac{1}{4 \a} \left( -(k - m) + 2 \sum_{j = m}^{k - 1} \de_j \right)}, &  \quad m < k.
                \end{cases}
            \end{equation*}
            First, we show that $N \e_N$ must be bounded. Suppose that it is not; this means that there exists a sequence $N_j$ such that $N_j \e_{N_j} \to \infty$. By the definition of $\e_N$ we can find $n_j$ such that
            \begin{equation*}
                \left| \frac{1}{N_j} \sum_{k = n_j}^{n_j + N_j - 1} \de_k \right| \ge \frac 1 2 + \e_{N_j} - \frac{1}{N_j}.        
            \end{equation*}
            Considering $m_j = n_j$ and $k_j = n_j + N_j$, we get
            \begin{equation*}
                \max(|T_{m_jk_j}|, |T_{k_jm_j}|) \gtrsim e^{\frac{1}{4 \a} \left( - N_j + 2 |\sum_{r = n_j}^{n_j + N_j - 1} \de_r| \right)} \ge e^{\frac{N_j \e_{N_j} - 1}{2 \a}} \to \infty. 
            \end{equation*}
            This means that in the matrix $(T_{mk})_{m, k \ge 0}$ there is a sequence of elements that grows unboundedly, in contradiction to the fact that $T$ must be a continuous operator on $\ell^p(\N_0)$. We conclude that $N \e_N$ has to be bounded.
            It remains to see that we cannot have $\e_N \ge 0$ for all $N$. Consider a sequence $N_j \to \infty$, we can find $n_j$ such that
            \begin{equation*}
                \left| \frac{1}{N_j} \sum_{k = n_j}^{n_j + N_j - 1} \de_k \right| \ge \frac 1 2 + \e_{N_j} - \frac{1}{N_j}.  
            \end{equation*}
            If we have
            \begin{equation*}
                \frac{1}{N_j} \sum_{k = n_j}^{n_j + N_j - 1} \de_k \ge \frac 1 2 + \e_{N_j} - \frac{1}{N_j},
            \end{equation*}
            then for $0 < K < N_j$ we get
            \begin{equation*}
                \sum_{r = n_j + K}^{n_j + N_j - 1} \de_r \ge \sum_{r = n_j}^{n_j + N_j - 1} \de_r - \left| \sum_{r = n_j}^{n_j + K - 1} \de_r \right| \ge \frac{N_j - K}{2} + N_j \e_{N_j} - K \e_K - 1.
            \end{equation*}
            Thus,
            \begin{equation*}
                |T_{mk}| \gtrsim e^{\frac{1}{2 \a} \left( N_j \e_{N_j} - K \e_{K} - 1 \right)} \gtrsim 1,
            \end{equation*}
            for $k = n_j + N_j$ and $m = n_j + K$ with $0 < K < N_j$, i.e., we get a part of the column of $(T_{mk})_{m, k \ge 0}$ with elements bounded away from $0$. Similarly, if 
            \begin{equation*}
                \frac{1}{N_j} \sum_{k = n_j}^{n_j + N_j - 1} \de_k \le -\frac 1 2 - \e_{N_j} + \frac{1}{N_j},
            \end{equation*}
            we get 
            \begin{equation*}
                |T_{mk}| \gtrsim e^{\frac{1}{2 \a} \left( N_j \e_{N_j} - K \e_{K} - 1 \right)} \gtrsim 1,
            \end{equation*}
            for $m = n_j + N_j$ and $k = n_j + K$ with $0 < K < N_j$, i.e., we get a part of the row of $(T_{mk})_{m, k \ge 0}$ with elements bounded away from $0$. In conclusion, we get increasingly long parts of the rows and columns of $(T_{mk})_{m, k \ge 0}$ with coefficients uniformly bounded away from $0$, which contradicts the continuity of $T$ on $\ell^p(\N_0)$. This completes the proof of Theorem \ref{thrm: CIS characterisation for one-sided}. \hfill \qedsymbol

\section{Proof of Theorem \ref{thrm: CIS characterisation for two-sided}} \label{sect: Theorem two-sided}
    The proof of this theorem follows the same path as the proof of Theorem \ref{thrm: CIS characterisation for one-sided}. We go over the steps of the proof, highlight the differences, and dive into the details only if the difference is major. Recall that in the two-sided context $\f(z)$ stands for $\a \log^2|z|$.
    \subsection{Sufficiency}
        \subsubsection{Completeness}
            We prove by contradiction that any $\L$ satisfying properties \ref{pr: 1 two-sided}-\ref{pr: 3 two-sided} is a uniqueness set for $\F^p_{\a}$. 
            Suppose that there is a non-zero function $f \in \F^p_{\a}$ such that $f|_{\L} = 0$. As before, we have $f(z) = h(z)G_{\L}(z)$ except that now with some $h \in \Hol(\C_0)$ and the canonical product
            \begin{equation} \label{eq: canonical product two-sided}
                G_{\L}(z) = \prod_{m \ge 1} \left( 1 - \frac{\l_{-m}}{z} \right) \prod_{k \ge 0} \left( 1 - \frac{z}{\l_k} \right).
            \end{equation}
            Lemma \ref{lem: product estimate fine} with property \ref{pr: 3 two-sided} gives us 
            \begin{equation*}
                |G_{\L}(z)| \gtrsim \frac{e^{\f(z)} \dist(z, \L)}{|z|^{1 + 2/p - \e}},
            \end{equation*}
            and, estimating $f$ with Lemma \ref{lem: evaluation estimate}, we get
            \begin{equation*}
                |h(z)| \lesssim \frac{|z|^{1 - \e}}{\dist(z, \L)}.
            \end{equation*}
            Since $h$ is not entire, this time we will have to consider two collections of contours going to $0$ and $\infty$ respectively. We consider contours $C_n, C_n'$, $n \ge 0$ around $0$ such that $C_n \cap e^n \D = \emptyset$ and $C_n' \subset e^{-n} \D, \ e^{- n - 1} \D \cap C'_n = \emptyset$, on which an estimate $\dist(z, \L) \gtrsim |z|$ holds. First, let us show that $h$ does not have a singularity at $0$. Consider $z: e^{-n-1} \le |z| < e^{-n}$, by the maximum principle we get
        \begin{equation*}
            |h(z)| \le \max_{w \in C'_{n - 1} \cup C'_{n + 1}} |h(w)| \lesssim \max_{w \in C'_{n - 1} \cup C'_{n + 1}} |w|^{-\e} \le e^{\e (n + 1)} \asymp |z|^{-\e}.
        \end{equation*}
        This implies that $zh(z)$ is analytic at $0$ and vanishes there, which in turn means that $h(z)$ is analytic at $0$.        
        
        Now that we know that $h$ is entire, the end of the proof is exactly the same. For $z \in \C$, by the maximum principle, for large enough $n$ we get
        \begin{equation*}
            |h(z)| \lesssim \max_{w \in C_n} \frac{|w|^{\e}}{\dist(w, \L)} \lesssim e^{-\e n}.
        \end{equation*}
        Letting $n$ tend to $\infty$, we conclude that $h(z) = 0$, which implies $f = 0$, a contradiction.
            
        \subsubsection{Base sequence}
            In this case, the base unperturbed sequence and its generating functions are
            \begin{equation*}
                \G = ( \g_k )_{k \in \Z} = \left( e^{\frac{k + 2/p}{2\a}} \right)_{k \in \Z},
            \end{equation*}
            \begin{equation*}
                G_{\G}(z) = \prod_{m \ge 1} \left( 1 - \frac{\g_{-m}}{z} \right) \prod_{k \ge 0} \left( 1 - \frac{z}{\g_k} \right).
            \end{equation*}
            \begin{prop} \label{prop: Gamma CIS two-sided}
                The sequence $\G$ is complete interpolating for $\F^p_{\a}$. Moreover, it stays complete interpolating even if we change the phases of $( \g_k )_{k \in \Z}$ in any way.
            \end{prop}
            \begin{proof}
                We again only need to prove that $\G$ is interpolating for $\F^p_{\a}$.
                For $k \in \Z$ set
                \begin{equation*}
                    g_k(z) = \frac{G_{\G}(z)}{G'_{\G}(\g_k)(z - \g_k)}.
                \end{equation*}
                We prove that 
                \begin{equation} \label{eq: interpolation formula for gamma two-sided}
                    f(z) = \sum_{k \in \Z} c_k |\g_k|^{-2/p}e^{\f(\g_k)} g_k(z)
                \end{equation}
                belongs to $\F^p_{\a}$ and interpolates an arbitrary sequence $(c_k)_{k \in \Z} \in \ell^p(\Z)$, i.e., $R_{\G} f = c$. We estimate $|G_{\G}|$ using Lemma \ref{lem: product estimate fine}:
                \begin{equation} \label{eq: Generating for Gamma estimate two-sided}
                    |G_{\G}(z)| \asymp \frac{e^{\f(z)} \dist(z, \G)}{|z|^{1/2 + 2/p}}, \quad z \in \C_0,
                \end{equation}
                \begin{equation} \label{eq: Generating for Gamma derivates two-sided}
                    |G'_{\G}(\g)| \asymp \frac{e^{\f(\g)}}{|\g|^{1/2 + 2/p}}, \quad \g \in \G.
                \end{equation}
                As before $f$ is a well-defined entire function by \eqref{eq: Generating for Gamma derivates two-sided}, $f(\g_k) = c_k$. Thus, it remains to show that $f \in \F^p_{\a}$. Similarly to the one-sided case, we separate two cases: $0 < p < 1$ and $1 \le p \le \infty$. For $0 < p < 1$
                \begin{equation*}
                    \|f\|^p_{\F^p_{\a}} = \int_{\C_0} |f(z)|^p e^{-p\f(z)} dA(z) \le \sum_{k \in \Z} |c_k|^p \int_{\C_0} |\g_k|^{-2}e^{p\f(\g_k)} |g_k(z)|^p e^{-p\f(z)} dA(z).
                \end{equation*}
                Thus, to show $f \in \F^p_{\a}$ it is sufficient to prove
                \begin{equation*}
                    I_k :=  \int_{\C_0} |\g_k|^{-2}e^{p\f(\g_k)} |g_k(z)|^p e^{-p\f(z)} dA(z) \lesssim 1.    
                \end{equation*}
                Using \eqref{eq: Generating for Gamma estimate two-sided} and \eqref{eq: Generating for Gamma derivates two-sided} to estimate $g_k(z)$ we get
                \begin{equation*}
                    I_k \asymp \int_{\C_0} \frac{\dist(z, \G)^p |\g_k|^{p/2}}{|z|^{p/2 + 2} |z - \g_k|^p} dA(z).
                \end{equation*}
                We partition $\C_0$ into two parts $D_1 = \{z: 0 < |z| < |\g_k|/2 \}$, $D_2 = \{z: |\g_k|/2 \le |z| \}$, which decomposes $I_k$ into the sum of respective integrals $I_k = I_k^{(1)} + I_k^{(2)}$. We then estimate each of them separately:
                \begin{align*}
                    I_k^{(1)} &\asymp \int_{0 < |z| < |\g_k|/2} \frac{|z|^p |\g_k|^{p/2}}{|z|^{p/2 + 2} |\g_k|^p} dA(z) \asymp |\g_k|^{-p/2} \int_0^{|\g_k|/2} r^{p/2 - 2} r dr \lesssim 1, \\
                    I_k^{(2)} &\lesssim \int_{|\g_k|/2 \le |z|} \frac{|\g_k|^{p/2}}{|z|^{p/2 + 2}} dA(z) \asymp |\g_k|^{p/2} \int_{|\g_k|/2}^{\infty} r^{-p/2 - 2} r dr \lesssim 1. 
                \end{align*} 
                We conclude that $I_k \lesssim 1$, and so $f \in \F^p_{\a}$, which completes the $0 < p < 1$ case.

                For $p \ge 1$, we bound $|f(z)|$ as
                \begin{equation*}
                    |f(z)|e^{-\f(z)} \lesssim  \sum_{k \in \Z} |c_k|  \frac{\dist(z, \G)|\g_k|^{1/2}}{|z|^{1/2 + 2/p} |z - \g_k|}.
                \end{equation*}
                The fact that
                \begin{equation*}
                    \sum_{k \in \Z} \frac{\dist(z, \G)|\g_k|^{1/2}}{|z|^{1/2} |z - \g_k|} \asymp 1
                \end{equation*}
                similarly follows by considering the $d_{\log}$-closest to $z$ point $\g_n$ and separating the sum into two parts $k \le n$ and $k > n$. For $p = \infty$, we immediately get $f \in \F^p_{\a}$. For $1 \le p < \infty$, we use the Jensen's inequality to estimate $\|f\|^p_{\F^p_{\a}}$:
                \begin{equation*}
                    \|f\|^p_{\F^p_{\a}} = \int_{\C_0} |f(z)|^pe^{-p\f(z)} dA(z) \lesssim \sum_{k \in \Z} |c_k|^p \int_{\C_0 } \frac{\dist(z, \G)|\g_k|^{1/2}}{|z|^{5/2}|z - \g_k|} dA(z).
                \end{equation*}
                We define
                \begin{equation*}
                    I_k = \int_{\C_0 } \frac{\dist(z, \G)|\g_k|^{1/2}}{|z|^{5/2} |z - \g_k|} dA(z),
                \end{equation*}
                and prove $I_k \lesssim 1$. In a similar manner, we partition $\C_0$ into two parts $D_1 = \{z: 0 < |z| < |\g_k|/2 \}$, $D_2 = \{z: |\g_k|/2 \le |z| \}$, and estimate $\frac{\dist(z, \G)}{|z - \g_k|} \lesssim \frac{|z|}{|\g_k|}$ on $D_1$ and $\frac{\dist(z, \G)}{|z - \g_k|} \lesssim 1$ on $D_2$. Integrating separately over $D_1$ and $D_2$ with these estimates we get $I_k \lesssim 1$. This completes the proof in the case $1 \le p \le \infty$.
                
                We have proven that $\G$ is complete interpolating for $\F^p_{\a}$. Note that adding phases to $\G$ does not affect the proof.         
            \end{proof} 

        \subsubsection{Interpolation} 
            We prove that $\L$ satisfying properties \ref{pr: 1 two-sided}-\ref{pr: 3 two-sided} is also interpolating for $\F^p_{\a}$. Since $\L$ is $d_{\log}$-separated, by Remark \ref{rmk: evaluation estimate} the restriction operator $R_{\L}$ is bounded. Consider 
            \begin{equation*}
                g_k(z) = \frac{G_{\L}(z)}{G'_{\L}(\l_k)(z - \l_k)}.
            \end{equation*}
            By Lemma \ref{lem: product estimate fine} and property \ref{pr: 3 two-sided} we have
            \begin{align*}  
                \frac{e^{\f(z)} \dist(z, \L)}{|z|^{1 + 2/p - \e}} \lesssim |G_{\L}(z)| \lesssim \frac{e^{\f(z)} \dist(z, \L)}{|z|^{2/p + \e}}, \quad z \in \C_0, 
                \\
                e^{\f(\l)} |\l|^{-1 - 2/p + \e} \lesssim |G'_{\L}(\l)| \lesssim e^{\f(\l)} |\l|^{- 2/p - \e}, \quad \l \in \L. 
            \end{align*}
            This implies $g_k \in \F^p_{\a}$, so we can interpolate finite sequences $c \in \ell^p(\Z)$ by finite linear combinations as follows:
            \begin{equation*}
                f(z) = \sum_{k \in \Z} c_k |\l_k|^{-2/p}e^{\f(\l_k)} g_k(z).
            \end{equation*}
            It is sufficient to prove stability on finite sequences. This is equivalent to proving the continuity of the Hilbert transform-type operator $T: \ell^p(\Z) \to \ell^p(\Z)$
            \begin{equation} \label{eq: T two-sided}
                (Tc)_m = \sum_{k \in \Z} \left| \frac{\g_m}{\l_k} \right|^{2/p} e^{\f(\l_k) - \f(\g_m)} \frac{G_{\L}(\g_m)}{G'_{\L}(\l_k)(\g_m - \l_k)} c_k = \sum_{k \in \Z} T_{mk} c_k, 
            \end{equation}
            \begin{equation*}
                |T_{mk}| = \left| \frac{\g_m}{\l_k} \right|^{2/p} e^{\f(\l_k) - \f(\g_m)} \left| \frac{G_{\L}(\g_m)}{G'_{\L}(\l_k)(\g_m - \l_k)} \right|.
            \end{equation*}
            Lemma \ref{lem: product estimates coarse} and property \ref{pr: 2 two-sided} give us
            \begin{align*}
                |G_{\L}(\g_m)| & \asymp \frac{\dist(\g_m, \L)}{|\g_m|} \prod_{j = 0}^{m - 1} \frac{|\g_m|}{|\l_j|}, \quad m \ge 0, \\
                |G_{\L}(\g_{-m})| & \asymp \frac{\dist(\g_{-m}, \L)}{|\g_{-m}|} \prod_{j = 1}^{m} \frac{|\l_{-j}|}{|\g_{-m}|}, \quad m \ge 1, \\
                |G'_{\L}(\l_k)| & \asymp \frac{1}{|\l_k|} \prod_{j = 0}^{k - 1} \frac{|\l_k|}{|\l_j|}, \quad k \ge 0, \\
                |G'_{\L}(\l_{-k})| & \asymp \frac{1}{|\l_{-k}|} \prod_{j = 1}^{k} \frac{|\l_{-j}|}{|\l_{-k}|}, \quad k \ge 1.
            \end{align*}
            Thus, we estimate
            \begin{equation*}
                |T_{mk}| \asymp e^{\f(\l_k) - \f(\g_m)} \left| \frac{\g_m}{\l_k} \right|^{2/p - 1} \frac{\dist(\g_m, \L)}{|\g_m - \l_k|} |\g_m|^m |\l_k|^{-k} P_{mk} ,
            \end{equation*}
            where 
            \begin{equation*}
                P_{mk} =
                \begin{cases}
                    \prod_{j = m}^{k - 1} |\l_j|, & \quad m < k, \\
                    1, & \quad m = k, \\
                    \prod_{j = k}^{m - 1} |\l_j|^{-1}, & \quad m > k.
                \end{cases}
            \end{equation*}
            We choose phases for $\G$ so that $\dist(\g_m, \L) \asymp |\g_m|$ and get
            \begin{equation} \label{eq: Tmk estimate two-sided}
                |T_{mk}| \asymp 
                \begin{cases}
                    e^{\frac{1}{4 \a} \left( -(k - m) + 2 \sum_{j = m}^{k - 1} \de_j \right)}, & \quad m < k, \\
                    1, & \quad m = k, \\
                    e^{\frac{1}{4 \a} \left( -(m - k) - 2 \sum_{j = k}^{m - 1} \de_j \right)}, & \quad m > k.
                \end{cases}
            \end{equation}
            Due to property \ref{pr: 3 two-sided} we obtain
            \begin{equation*}
                |T_{mk}| \lesssim e^{- \frac{\e}{4 \a} |m - k|}, \quad m, k \in \Z.
            \end{equation*}
            Thus, the operator $T$ is continuous from $\ell^p(\Z)$ onto itself, which completes the proof.

    \subsection{Necessity}
        Suppose that $\L$ is complete interpolating. 

        \subsubsection{Property \ref{pr: 1 two-sided}}
            We only need the second part of the argument for the two-sided case. Suppose that $\L$ is not $d_{\log}$-separated and let us get a contradiction. 
            We consider a sequence of distinct point pairs $\l_{n_j}, \l_{m_j} \in \L$ such that $d_{\log}(\l_{n_j}, \l_{m_j}) \to 0$. Consider $f_j$ to be the solution to the interpolation problem $f(\l) = e^{\f(\l)} |\l|^{-2/p} \de_{\l, \l_{n_j}}$, so that $\|f_j\|_{\F^p_{\a}} \lesssim 1$. Evaluating $f_j$ at the points $\l_{n_j}, \l_{m_j}$ and taking $j \to \infty$, we arrive at a contradiction to Lemma \ref{lem: more accurate estimates}.
            
            Hence, we have property \ref{pr: 1 two-sided} and so we can enumerate $\L = ( \l_k )_{k \in \Z}$ so that $|\l_k| \le |\l_{k + 1}|$ and write
            \begin{equation*}
                \l_k = e^{\frac{k + 2/p + \de_k}{2\a}} e^{i \theta_k}, \quad \de_k, \theta_k \in \R.    
            \end{equation*}
            Finally, property \ref{pr: 1 two-sided} allows us to choose phases for $\G$ in such a way that $d_{\log}(\l_k, \G) \asymp 1$.

        \subsubsection{Property \ref{pr: 2 two-sided}}
            For property \ref{pr: 2 two-sided}, we need to show that $(\de_k)_{k \in \Z} \in \ell^{\infty}(\Z)$. Since $\L$ is complete interpolating, that implies that for any $k \in \Z$ there exists $g_k \in \F^p_{\a}$ that uniquely interpolates $g_k(\l_j) = \de_{k,j}, \ j \in \Z$, moreover $\|g_k\|^p_{\F^p_{\a}} \asymp |\l_k|^2 e^{-p \f(\l_k)}$. The major difference from the one-sided case is that we can make an arbitrary integer shift in the enumeration of $\l_n, \ n \in \Z$ and still get a sequence growing in moduli. In particular, we get infinitely many different canonical products \eqref{eq: canonical product two-sided} for each such shift. In contrast to \eqref{eq: biorthogonal element one-sided}, it is not clear which canonical product gives us each $g_k, \ k \in \Z$. 
            We claim that the right choice of canonical product exists and it is the same for all $k \in \Z$.
            \begin{prop}
                There exists $M \in \Z$ such that
                \begin{equation} \label{eq: generating function enumeration shift}
                    g_k(z) = \left( \frac{z}{\l_k} \right)^M \frac{G_{\L}(z)}{G'_{\L}(\l_k)(z - \l_k)}, \quad k \in \Z.         
                \end{equation}    
            \end{prop} 
            \begin{proof}
                Note that by the complete interpolating property of $\L$ we have
                \begin{equation*}
                    g_k(z) = C_k \frac{z - \l_0}{z - \l_k} g_0(z), \quad k \in \Z,
                \end{equation*}
                for some constants $C_k \in \C_0$. This means that it is sufficient to guarantee \eqref{eq: generating function enumeration shift} for a single $k$, e.g., $k = 0$. Therefore, it suffices to prove that there exists $M \in \Z$ such that
                \begin{equation*}
                    g_0(z) = \left( \frac{z}{\l_0} \right)^M \frac{G_{\L}(z)}{G'_{\L}(\l_0)(z - \l_0)}.
                \end{equation*}
                Put
                \begin{equation*}
                    f(z) = \frac{G_{\L}(z)}{G'_{\L}(\l_0)(z - \l_0)}.
                \end{equation*}
                Consider $g_0(e^{w})$ and $f(e^w)$. These are two entire functions with the same zeros. Moreover, due to Lemma \ref{lem: evaluation estimate}, $g_0(e^w)$ is of order $2$. Similarly, because $\L$ is $d_{\log}$-separated, by Lemma \ref{lem: product estimates coarse} $|G_{\L}(z)| \lesssim |G_{\G_{\b}}(z)|$ where $\G_{\b} = (e^{\frac{k}{2 \b}})_{k \in \Z}$ for a large enough $\b > 0$. Hence, by Lemma \ref{lem: product estimate fine}, $f(e^w)$ is also of order $2$.
                Thus, by the Hadamard factorization theorem, we have
                \begin{equation*}
                    g_0(e^w) = e^{aw^2 + bw + c} f(e^w), 
                \end{equation*}
                for some $a, b, c \in \C$. Note that $g_0(e^w)$ and $f(e^w)$ are $2 \pi i$-periodic. This means that $e^{aw^2 + bw + c}$ is also $2 \pi i$-periodic, i.e., 
                \begin{equation*}
                    e^{aw^2 + bw + c} = e^{a(w + 2 \pi i)^2 + b(w + 2 \pi i) + c}, \quad w \in \C.
                \end{equation*}
                This is equivalent to 
                \begin{equation*}
                    e^{a(4 \pi i w - 4 \pi^2) + 2 \pi i b} = 1, \quad w \in \C.
                \end{equation*}
                It follows that we must have $a = 0$ and $b \in \Z$. Thus, we can take $M = b$ and $C = e^c \in \C_0$ to get
                \begin{equation*}
                    g_0(e^w) = C (e^w)^M f(e^w), \quad w \in \C.
                \end{equation*}
                Hence, we have $g_0(z) = C z^M f(z)$. It remains to note that $C = \l_0^{-M}$ since $g_0(\l_0) = f(\l_0) = 1$. This completes the proof.
            \end{proof}
            Note that \eqref{eq: generating function enumeration shift} is equivalent to
            \begin{equation*}
                g_k(z) = \frac{G_{\L}(z)}{G'_{\L}(\l_k)(z - \l_k)}, \quad k \in \Z,         
            \end{equation*}
            if we shift the enumeration of $\L$ by $M$.

            Let us return to the proof of property \ref{pr: 2 two-sided}.
            Suppose, for contradiction, that $\de_k$ is unbounded. Then there exists a subsequence $k_j$ such that $|\de_{k_j}| \to \infty$. If we let $\g_{m_j} \in \G$ be the $d_{\log}$-closest point to $\l_{k_j}$, then we get $|m_j - k_j| \to \infty$. Our next step is to study the behavior of $|g_{k_j}(\g_{m_j})|$. It follows from Lemma \ref{lem: product estimates coarse} and property \ref{pr: 1 two-sided} that
            \begin{align*}
                |G_{\L}(\g_{m_j})| & \asymp \prod_{l = 0}^{k_j - 1} \frac{|\g_{m_j}|}{|\l_l|}, \\
                |G'_{\L}(\l_{k_j})| & \asymp \frac{1}{|\l_{k_j}|} \prod_{l = 0}^{k_j - 1} \frac{|\l_{k_j}|}{|\l_l|},
            \end{align*}
            where for $k_j = 0$ we take the products to be $1$ and for $k_j < 0$ we invert the multipliers and take the product from $l = k_j$ to $l = -1$.
            From here it follows that
            \begin{equation*}
                |g_{k_j}(\g_{m_j})| \asymp \left| \frac{\g_{m_j}}{\l_{k_j}} \right|^{k_j} \asymp e^{\frac{1}{4 \a} 2k_j (m_j - k_j - \de_{k_j})}.
            \end{equation*}
            At the same time, by Lemma \ref{lem: evaluation estimate},
            \begin{equation*}
                |g_{k_j}(\g_{m_j})| \lesssim \frac{e^{\f(\g_{m_j})}}{|\g_{m_j}|^{2/p}} \| g_{k_j} \|_{\F^p_{\a}} \asymp e^{\f(\g_{m_j}) - \f(\l_{k_j})} \asymp e^{\frac{1}{4 \a}\left( (m_j + 2/p)^2 - (k_j + \de_{k_j} + 2/p)^2 \right)}.
            \end{equation*}
            Combining these two estimates, we get
            \begin{equation} \label{eq: de to infty contradiction two-sided}
                e^{\frac{1}{4 \a} (m_j - k_j - \de_{k_j})(k_j - m_j - \de_{k_j} - 4/p)} \lesssim 1.
            \end{equation}
            By choosing a subsequence of $\de_{k_j}$, we can assume that $\de_{k_j} \to \infty$ or $\de_{k_j} \to -\infty$. In the case $\de_{k_j} \to \infty$ we get $k_j - m_j - \de_{k_j} - 4/p \to -\infty$. So, it is enough to decrease all $m_j$ by one sufficiently large integer to get $m_j - k_j - \de_{k_j} \le -1$ to arrive at a contradiction to \eqref{eq: de to infty contradiction two-sided}. Similarly, in the case $\de_{k_j} \to -\infty$, we get $k_j - m_j - \de_{k_j} - 4/p \to \infty$. So, increasing all $m_j$ by one sufficiently large integer to get $m_j - k_j - \de_{k_j} \ge 1$ gives us a contradiction to \eqref{eq: de to infty contradiction two-sided}. We conclude that $\de_k$ has to be bounded, which is exactly property \ref{pr: 2 two-sided}.

        \subsubsection{Property \ref{pr: 3 two-sided}}
            Finally, let us prove property \ref{pr: 3 two-sided}. Suppose that it does not hold, i.e., for every $N > 0$ 
            \begin{equation*}
                \sup_{n \in \Z} \left| \frac{1}{N} \sum_{k = n}^{n + N - 1} \de_k \right| = \frac 1 2 + \e_N,     
            \end{equation*}
            with $\e_N \ge 0$. Let us again consider operator $T$ as in \eqref{eq: T two-sided}.
            Since $\L$ is complete interpolating, $T$ must be continuous. Moreover, because we established property \ref{pr: 2 two-sided}, the estimate \eqref{eq: Tmk estimate two-sided} holds true, i.e.,
            \begin{equation*}
                |T_{mk}| \asymp 
                \begin{cases}
                    1, & \quad m = k, \\
                    e^{\frac{1}{4 \a} \left( -(m - k) - 2 \sum_{j = k}^{m - 1} \de_j \right)}, & \quad m > k, \\
                    e^{\frac{1}{4 \a} \left( -(k - m) + 2 \sum_{j = m}^{k - 1} \de_j \right)}, & \quad m < k.
                \end{cases}
            \end{equation*}
            First, we show that $N \e_N$ must be bounded. Suppose that it is not; this means that there exists a sequence $N_j$ such that $N_j \e_{N_j} \to \infty$. By the definition of $\e_N$ we can find $n_j$ such that
            \begin{equation*}
                \left| \frac{1}{N_j} \sum_{k = n_j}^{n_j + N_j - 1} \de_k \right| \ge \frac 1 2 + \e_{N_j} - \frac{1}{N_j}.        
            \end{equation*}
            Considering $m_j = n_j$ and $k_j = n_j + N_j$, we get
            \begin{equation*}
                \max(|T_{m_jk_j}|, |T_{k_jm_j}|) \gtrsim e^{\frac{1}{4 \a} \left( - N_j + 2 |\sum_{r = n_j}^{n_j + N_j - 1} \de_r| \right)} \ge e^{\frac{N_j \e_{N_j} - 1}{2 \a}} \to \infty. 
            \end{equation*}
            This means that in the matrix $(T_{mk})_{m, k \in \Z}$ there is a sequence of elements that grows unboundedly, in contradiction to the fact that $T$ must be a continuous operator on $\ell^p(\Z)$. We conclude that $N \e_N$ has to be bounded.
            It remains to see that we cannot have $\e_N \ge 0$ for all $N$. Consider a sequence $N_j \to \infty$, we can find $n_j$ such that
            \begin{equation*}
                \left| \frac{1}{N_j} \sum_{k = n_j}^{n_j + N_j - 1} \de_k \right| \ge \frac 1 2 + \e_{N_j} - \frac{1}{N_j}.  
            \end{equation*}
            If we have
            \begin{equation*}
                \frac{1}{N_j} \sum_{k = n_j}^{n_j + N_j - 1} \de_k \ge \frac 1 2 + \e_{N_j} - \frac{1}{N_j},
            \end{equation*}
            then for $0 < K < N_j$ we get
            \begin{equation*}
                \sum_{r = n_j + K}^{n_j + N_j - 1} \de_r \ge \sum_{r = n_j}^{n_j + N_j - 1} \de_r - \left| \sum_{r = n_j}^{n_j + K - 1} \de_r \right| \ge \frac{N_j - K}{2} + N_j \e_{N_j} - K \e_K - 1.
            \end{equation*}
            Thus,
            \begin{equation*}
                |T_{mk}| \gtrsim e^{\frac{1}{2 \a} \left( N_j \e_{N_j} - K \e_{K} - 1 \right)} \gtrsim 1,
            \end{equation*}
            for $k = n_j + N_j$ and $m = n_j + K$ with $0 < K < N_j$, i.e., we get a part of the column of $(T_{mk})_{m, k \in \Z}$ with elements bounded away from $0$. Similarly, if 
            \begin{equation*}
                \frac{1}{N_j} \sum_{k = n_j}^{n_j + N_j - 1} \de_k \le -\frac 1 2 - \e_{N_j} + \frac{1}{N_j},
            \end{equation*}
            we get 
            \begin{equation*}
                |T_{mk}| \gtrsim e^{\frac{1}{2 \a} \left( N_j \e_{N_j} - K \e_{K} - 1 \right)} \gtrsim 1,
            \end{equation*}
            for $m = n_j + N_j$ and $k = n_j + K$ with $0 < K < N_j$, i.e., we get a part of the row of $(T_{mk})_{m, k \in \Z}$ with elements bounded away from $0$. In conclusion, we get increasingly long parts of the rows and columns of $(T_{mk})_{m, k \in \Z}$ with coefficients uniformly bounded away from $0$, which contradicts the continuity of $T$ on $\ell^p(\Z)$. This completes the proof of Theorem \ref{thrm: CIS characterisation for two-sided}. \hfill \qedsymbol

    \subsection*{Acknowledgment}
        The author expresses his gratitude to his advisors, Evgeny Abakumov and Alexander Borichev, for their constant support and guidance. He also thanks Yurii Belov for helpful comments and remarks.  
    
\printbibliography
        
\end{document}